\documentclass[11pt]{amsart}
\usepackage{url}
\usepackage{amsmath,amsthm,amsfonts,amssymb,latexsym}

\newtheorem{theorem}{Theorem}%[section]
\newtheorem{proposition}[theorem]{Proposition}
\newtheorem{lemma}[theorem]{Lemma}
\newtheorem{remark}[theorem]{Remark}

\newtheorem{corollary}[theorem]{Corollary}

\theoremstyle{definition}

\newtheorem{example}[theorem]{Example}

\newcommand{\IR}{\mathbb R}
\newcommand{\U}{\mathcal U}
\newcommand{\w}{\omega}

\newcommand{\M}{\mathcal{M}}
\newcommand{\F}{\mathcal{F}}
\newcommand{\cc}{\mathfrak c}
\newcommand{\uhr}{\upharpoonright}

\newcommand{\cof}{\mathrm{cof}}

\title[Non-meager free sets for meager relations]{Non-meager free sets for meager relations on Polish spaces}

\author{Taras Banakh,   Lyubomyr Zdomskyy}

\address{Department of Mathematics, Ivan Franko National University of Lviv, Ukraine; and\newline
Instytut Matematyki, Jan Kochanowski University in Kielce, Poland.}
\email{t.o.banakh@gmail.com}
\urladdr{http://www.franko.lviv.ua/faculty/mechmat/Departments/Topology/bancv.html}

\address{Kurt G\"odel Research Center for Mathematical Logic,
University of Vienna, W\"ahringer Stra\ss e 25, A-1090 Wien,
Austria.}
\email{lzdomsky@gmail.com}
\urladdr{http://www.logic.univie.ac.at/\~{}lzdomsky/}

\thanks{The first author has been partially financed by NCN grant  DEC-2011/01/B/ST1/01439.
The second author is a recipient of an APART-fellowship of the Austrian Academy of Sciences.}

\begin{document}
\begin{abstract} We prove that
for each meager relation $E\subset X\times X$ on a Polish space $X$ there is a
nowhere meager subspace $F\subset X$ which is $E$-free in the sense
that $(x,y)\notin E$ for any distinct points $x,y\in F$.
\end{abstract}

\maketitle

\section{Introduction}
This paper is devoted to the problem of finding non-meager free subsets
for meager relations on Polish spaces.
For a relation $E\subset X\times X$, a subset $F\subset X$ is called
 {\em $E$-free} if $(x,y)\notin E$ for any distinct points $x,y\in F$.
This is equivalent to saying that $F^2\cap E\subset\Delta_X$ where
$\Delta_X=\{(x,y)\in X^2:x=y\}$ is the diagonal of $X^2$.

The problem of finding ``large'' free sets for certain ``small''
relations was considered by many authors, see \cite{NPS}, \cite{SS}, \cite{Kubis},
 \cite{Geschke1}, \cite{Geschke2}. Observe that the classical
Mycielski-Kuratowski Theorem~\cite[18.1]{Ke} implies that for each meager
 relation $E\subset X^2$ on a perfect Polish space $X$ there is an
$E$-free perfect subset $F\subset X$. We recall that a subset of a Polish
space is {\em perfect} if it is closed and has no isolated points.
Nonetheless the following result seems to be new.

\begin{theorem}\label{main}
For each meager relation $E\subset X^2$ on a Polish space $X$ there is
an $E$-free nowhere meager subspace $B\subset X$.
Moreover, if the set of isolated points is not dense in $X$ then
$B$ may be chosen of any cardinality $\kappa\in [\mathrm{cof}(\mathcal M),\mathfrak c]$.
\end{theorem}

Let us recall that a subspace $A$ of a topological space $X$
\begin{itemize}
\item is {\em meager in $X$}, if $A$ can be written as a countable
union $A=\bigcup_{n\in\w}A_n$ of nowhere dense subsets of $X$;
\item is {\em nowhere meager in $X$}, if for any non-empty open set $U\subset X$ the intersection $U\cap A$ is not meager in $X$.
\end{itemize}
It is clear that a subset $A\subset X$ of a Polish space $X$ is nowhere meager if and only if $A$ is dense in $X$ and contains no open meager subspace. By definition,
$\mathrm{cof}(\mathcal M)$ is the minimal cardinality of a
collection $\mathcal X$ of meager subsets of the Baire space $\w^\w$ such that for every meager
$A\subset\w^\w$ there exists $X\in\mathcal X$ containing $A$. It is known \cite{Bla10} that $\cof(\M)=\mathfrak c$ under Martin's Axiom, and $\cof(\M)<\cc$ in some models of ZFC, see \cite{BJS}.

Theorem~\ref{main} will be proved in Section~\ref{main:pf}.
One of its applications is the existence of a first-countable
uniform Eberlein compact space which is not supercompact
(see \cite[5.2]{BKT}),
which was our initial motivation for considering free non-meager sets for meager relations.
The following simple example shows that the nowhere meager set $F$ in Theorem~\ref{main}
cannot have the Baire property. We recall that a subset $A$ of a topological space $X$ has {\em the Baire property} in $X$ if for some open set $U\subset X$ the symmetric difference $A\triangle U=(A\setminus U)\cup(U\setminus A)$ is meager in $X$.

\begin{example} For the nowhere dense relation
 $$E=\bigcup_{n\in\w}\{(x,y)\in\IR^2:|x-y|=2^{-n}\}\subset\IR\times\IR$$
on the real line $\IR$, each $E$-free subset $F\subset \IR$ with
the Baire property is meager.
\end{example}

\begin{proof} Assuming that $F$ is not meager, and using
 the Baire property of $F$, find a non-empty open subset $U\subset \IR$ such that
 $U\triangle F$ is meager and hence lies in some meager
 $F_\sigma$-set $M\subset \IR$.
Then $G=U\setminus M\subset F$ is a dense $G_\delta$-set in $U$.
By the Steinhaus-Pettis Theorem~\cite[9.9]{Ke},
the difference $G-G=\{x-y:x,y\in G\}$ is a neighborhood of zero
in $\IR$ and hence $2^{-n}\in G-G$ for some $n\in\w$.
Then any points $x,y\in G\subset F$ with $|x-y|=2^{-n}$ witness that
the set $F\ni x,y$ is not $E$-free.
\end{proof}

\begin{remark} {\rm By a classical result of Solovay \cite{Solovay},
there are models of ZF in which all subsets of the real line have
the Baire property. In such models each $E$-free subset for the
relation $E=\bigcup_{n\in\w}\{(x,y)\in\IR^2:|x-y|=2^{-n}\}$ is meager.
This means that the proof of Theorem~\ref{main}
must essentially use the Axiom of Choice.}
\end{remark}

\section{Some auxiliary results}\label{aux}

We recall \cite{BanZdo} that a family
$\mathcal F$ of infinite subsets of a countable set $X$ is called a \emph{semifilter},
if $A\in\mathcal F$ provided $F\subset^\ast A\subset X$ for some set $F\in\mathcal F$. Here $F\subset^* A$ means that $F\setminus A$ is finite.
Each semifilter on $X$ is contained in the semifilter $[X]^\w$ of all infinite subsets of $X$.
The semifilter $[X]^\w$ is a subset of the power set $\mathcal P(X)$ which can be identified with the Tychonoff product $2^X$ via characteristic functions. So, we can speak about topological properties of semifilters as subspaces of the compact Hausdorff space $\mathcal P(X)$.
According to Talagrand's characterization of meager semifilters on $\w$, a semifilter $\F$ on a countable set $X$ is meager (as a subset of $\mathcal P(X)$) if and only if $\F$ can be enlarged to a  $\sigma$-compact semifilter $\tilde\F\subset [X]^\w$. This characterization implies the following:

\begin{corollary}\label{tala} For any finite-to-one map $\phi:X\to Y$ between countable sets, a semifilter $\F\subset\mathcal P(X)$ is meager if and only if the semifilter $\phi[\F]=\{E\subset Y:\phi^{-1}(E)\in\F\}\subset\mathcal P(Y)$ is meager.
\end{corollary}

We recall that a map $f:X\to Y$ between two sets is called {\em finite-to-one} if for each $y\in Y$ the preimage $\psi^{-1}(y)$ is finite and non-empty. In particular, each monotone surjection $\psi:\w\to\w$ is finite-to-one.

A key ingredient of the proof of Theorem~\ref{main} in the following proposition.

\begin{proposition} \label{main_prop}
For any meager relation $E\subset 2^\w\times 2^\w$ on the Cantor cube $2^\w$ there is a family $(G_\alpha)_{\alpha<\mathfrak c}$ of nowhere meager subsets in $2^\w$ such that $(G_\alpha\times G_\beta)\cap E=\emptyset$ for any distinct ordinals $\alpha,\beta<\mathfrak c$.
\end{proposition}

\begin{proof} Using the fact that the points of the Cantor cube $2^\w$ can be identified with the branches of the binary tree $2^{<\w}=\bigcup_{n\in\w}2^n$, we can find a closed subset $\{A_\alpha\}_{\alpha<\mathfrak c}$ of $\mathcal P(\w)=2^\w$ which consists of infinite subsets of $\w$ and is almost disjoint in the sense that $A_\alpha\cap A_\beta$ is finite for any distinct ordinals $\alpha,\beta<\mathfrak c$. The compactness of $\{A_\alpha\}_{\alpha<\cc}$ in $2^\w$ implies the existence of a monotone surjection $\varphi:\w\to\w$ such that $\varphi(A_\alpha)=\w$ for all $\alpha<\cc$.

Fix any free ultrafilter $\U$ on $\w$ and for every $\alpha<\mathfrak c$ choose an ultrafilter $\U_\alpha$ on $\w$ extending the family $\{A_\alpha\cap \varphi^{-1}[U]:U\in\U\}$. The almost disjoint property of the family $\{A_\alpha\}_{\alpha<\mathfrak c}$ guarantees that $\w\setminus A_\alpha\in \U_\xi$ for any distinct ordinals $\alpha,\xi<\cc$.

\begin{lemma}\label{automat}
For every $\alpha<\cc$, the filter $$\F_\alpha=\mathcal P(\w\setminus A_\alpha)\cap \bigcap_{\alpha\ne\xi<\cc}\U_\xi$$
is non-meager in $\mathcal P(\w\setminus A_\alpha)$.
\end{lemma}

\begin{proof} By Corollary~\ref{tala}, the filter $\F_\alpha$ is not meager in $\mathcal P(\w\setminus A_\alpha)$ as its image $\varphi[\F_\alpha]=\{E\subset \w:\varphi^{-1}[E]\in\F_\alpha\}$ coincides with the ultrafilter $\U$ and hence is not meager in $\mathcal P(\w)$.
\end{proof}

Let $E\subset 2^\w\times 2^\w$ be a meager relation on $2^\w$.
By  \cite[Theorem~2.2.4]{BarJud95}, there exist a
 monotone surjection $\phi:\w\to\w$ and functions $f_0,f_1:\w\to 2$
such that
$$E\subset \big\{(g,g')\in 2^\w\times 2^\w\: : \: \forall^\infty n\in\w\:
 \big(g\uhr\phi^{-1}(n)\neq f_0\uhr\phi^{-1}(n)\big)\: \vee\: \big(g'\uhr\phi^{-1}(n)\neq f_1\uhr\phi^{-1}(n) \big) \big\}.
$$

For every ordinal
$\alpha<\cc$ consider the subset
$$
\begin{aligned}
G_\alpha= \big\{g\in 2^\w\: :\: &\exists X_0,X_1\in \U_\alpha\setminus\bigcup_{\alpha\ne\xi<\cc}\U_\xi\\
&\big(X_0\subset X_1) \wedge \big(g\uhr \phi^{-1}[X_0]=f_0\uhr \phi^{-1}[X_0]\big) \wedge \big(g\uhr \phi^{-1}[\w\setminus X_1]=f_1\uhr \phi^{-1}[\w\setminus X_1]\big) \big\}
\end{aligned}
$$in the Cantor cube $2^\w$.

\begin{lemma} \label{l1}
For every ordinal
$\alpha<\cc$ the set $G_\alpha$ is nowhere meager in $2^\w$.
\end{lemma}

\begin{proof}
Since $G_\alpha$ is closed under finite modifications of its elements, it is enough to show that
 $G_\alpha$ is non-meager in $2^\w$. Observe that $G_\alpha$  contains the set
$$
\begin{aligned}
G'_\alpha= \big\{g\in 2^\w\colon &\exists Y_0\in \U_\alpha \cap\mathcal P(A_\alpha) \;\;
\exists Y_1\in \mathcal P(\w\setminus A_\alpha)\setminus \bigcup_{\alpha\ne\xi<\cc}\U_\xi \\
&\big(g\uhr \phi^{-1}[Y_0]=f_0\uhr \phi^{-1}[Y_0]) \wedge \big(g\uhr \phi^{-1}[\w\setminus(A_\alpha\cup Y_1)]=f_1\uhr  \phi^{-1}[\w\setminus(A_\alpha\cup Y_1)]\big) \big\}.
\end{aligned}
$$
Indeed, if $g\in G_\alpha'$ is witnessed by $Y_0,Y_1$, then $X_0=Y_0$ and $X_1=A_\alpha\cup Y_1$
are witnessing that $g\in G_\alpha$.
Now $G_\alpha'$ may be written as the product $R_\alpha\times H_\alpha$, where
$$ R_\alpha=\big\{g\in 2^{\phi^{-1}[A_\alpha]}\;\colon \exists Y_0\in \U_\alpha \cap\mathcal P(A_\alpha) \;\;
\big(g\uhr \phi^{-1}[Y_0]=f_0\uhr \phi^{-1}[Y_0]\big) \big\} $$
and
$$
\begin{aligned}
H_\alpha= \big\{g\in 2^{\phi^{-1}[\w\setminus A_\alpha]}\;\colon
&\exists Y_1\in \mathcal P(\w\setminus A_\alpha)\setminus \bigcup_{\alpha\ne\xi<\cc}\U_\xi\\
&\big(g\uhr \phi^{-1}[\w\setminus(A_\alpha\cup Y_1)]=f_1\uhr  \phi^{-1}[\w\setminus(A_\alpha\cup Y_1)]\big) \big\}.
 \end{aligned}
$$
Thus it suffices to show that both $R_\alpha$ and $H_\alpha$ are non-meager.
By the homogeneity of $2^\w$ there is no loss of generality to assume that
$f_0\uhr \phi^{-1}[A_\alpha]\equiv 1$ and $f_1\uhr  \phi^{-1}[\w\setminus A_\alpha]\equiv 1$.

With $f_1$  as above  we see that $H_\alpha$ is simply the set of characteristic functions
of elements of the semifilter
\begin{eqnarray*}
 \mathcal H_\alpha= \big\{Z\subset \phi^{-1}[\w\setminus A_\alpha]\: :\: \exists Y_1\in \mathcal P(\w\setminus A_\alpha)\setminus \bigcup_{\alpha\ne\xi<\cc}\U_\xi
\;\;\;\;
  \big(\phi^{-1}[\w\setminus (A_\alpha\cup Y_1)]\subset Z\big)\big\}
\end{eqnarray*}
on $\phi^{-1}[\w\setminus A_\alpha]$.
Therefore
$$ \phi[\mathcal H_\alpha]= \big\{T\subset \w\setminus A_\alpha\: :\: \exists Y_1\in \mathcal P(\w\setminus A_\alpha)\setminus \bigcup_{\alpha\ne\xi<\cc}\U_\xi
\;\;\;\;  \big(\w\setminus (A_\alpha\cup Y_1)\subset T\big)\big\}. $$
Observe that  $Y_1\in \mathcal P(\w\setminus A_\alpha)\setminus \bigcup_{\alpha\ne\xi<\cc}\U_\xi$
iff  $\w\setminus (A_\alpha\cup Y_1)\in \bigcap_{\alpha\ne\xi<\cc}\U_\xi$,
and hence $\phi[\mathcal H_\alpha]$ is equal to the filter $\mathcal P(\w\setminus A_\alpha)\cap \bigcap_{\alpha\ne\xi<\cc}\U_\xi$ which is non-meager in $\mathcal P(\w\setminus A_\alpha)$ by Lemma~\ref{automat}, and consequently the filter
 $\mathcal H_\alpha$ is non-meager in $\mathcal P(\phi^{-1}[\w\setminus A_\alpha])$ by Corollary~\ref{tala}. In other words, $H_\alpha$ is a non-meager subset of $2^{\phi^{-1}[\w\setminus A_\alpha]}$.

The proof of the fact that $R_\alpha$ is non-meager is analogous. However, we present it
for the  sake of completeness.
With $f_0$  as above  we see that $R_\alpha$ is simply the set of characteristic functions
of elements of the semifilter
$$ \mathcal R_\alpha= \{Z\subset \phi^{-1}[A_\alpha]\: :\: \exists Y_0\in \mathcal P(A_\alpha)\cap \U_\alpha
\  \big(\phi^{-1}[Y_0]\subset Z\big)\}$$
on $\phi^{-1}[A_\alpha]$.
It follows that
$$ \phi[\mathcal R_\alpha]= \{T\subset  A_\alpha\: :\: \exists Y_0\in \mathcal P(A_\alpha)\cap \U_\alpha
\  \big(Y_0\subset T \big)\}= \mathcal P(A_\alpha)\cap \U_\alpha$$is a non-meager ultrafilter on $A_\alpha$ and then $\mathcal R_\alpha$ is a non-meager semifilter on $\phi^{-1}[A_\alpha]$ according to Corollary~\ref{tala}. Consequently, $R_\alpha$ is a non-meager subset of $2^{\phi^{-1}[A_\alpha]}$.
\end{proof}

\begin{lemma} For any distinct ordinals $\alpha,\beta<\cc$ we get $(G_\alpha\times G_\beta)\cap E=\emptyset$.
\end{lemma}

\begin{proof} Assume conversely that $(G_\alpha\times G_\beta)\cap E$ contains some pair $(g_\alpha,g_\beta)$. Fix sets $X^{\alpha}_0,X^{\alpha}_1$ and $X^{\beta}_0,X^{\beta}_1$
witnessing that $g_{\alpha}\in G_{\alpha}$ and $g_{\beta}\in G_{\beta}$, respectively.
The intersection $X^{\alpha}_0\cap (\w\setminus X^{\beta}_1)$ is infinite:
 otherwise $X^{\alpha}_0\subset^* X^{\beta}_1$ and  $X^{\beta}_1\in\U_{\alpha}$, which contradicts the definition of $G_{\beta}$.
Thus the set $X^\alpha_0\setminus X^\beta_1$ is infinite and for every $n\in X^\alpha_0\setminus X^\beta_1$ we get $g_{\alpha}\uhr \phi^{-1}(n)=f_0 \uhr \phi^{-1}(n) $
and $g_{\beta}\uhr \phi^{-1}(n)=f_1 \uhr \phi^{-1}(n)$, which implies $(g_{\alpha},g_{\beta})\not\in E$.
\end{proof}
This completes the proof of Proposition~\ref{main_prop}.
\end{proof}

Using the well-known fact that each perfect Polish space $X$ contains a dense $G_\delta$-subset homeomorphic to the space of irrationals $\w^\w$, we can generalize Proposition~\ref{main_prop} as follows.

\begin{proposition}\label{p2}
For any meager relation $E\subset X\times X$ on a perfect Polish space $X$ there is a family $(G_\alpha)_{\alpha<\mathfrak c}$ of nowhere meager subsets in $X$ such that $(G_\alpha\times G_\beta)\cap E=\emptyset$ for any distinct ordinals $\alpha,\beta<\mathfrak c$.
\end{proposition}

\section{Proof of Theorem~\ref{main}}\label{main:pf}

Let $E\subset X\times X$ be a meager
relation on a Polish space $X$.
If the set $D$ of isolated points is dense in $X$,
then $B=D$ is a required nowhere meager $E$-free subset of $X$.
 So, we assume that the set $D$ is not dense in $X$.
Then the open subspace $Y=X\setminus \bar D$ of $X$ is not empty and has no isolated points.
Let $\kappa\in [\mathrm{cof}(\mathcal M),\mathfrak c]$ be any cardinal.
By Proposition~\ref{p2}, there is a family $(G_\alpha)_{\alpha<\kappa}$
 of nowhere meager subsets in $Y$ such that $(G_\alpha\times G_\beta)\cap E=\emptyset$ for any distinct ordinals $\alpha,\beta<\kappa$.

Let $\U$ be a countable base of the topology  of $Y$ and $\mathcal X$ be a cofinal with respect to
inclusion family of meager subsets in $Y$ of size $\kappa$.
It is clear that the set $\U\times\mathcal X$ has cardinality $\kappa$  and hence can be enumerated as
$\U\times\mathcal X=\{(U_\alpha,X_\alpha):\alpha<\kappa\}$. Since the set $D$ is at most countable and
$E$ is meager in $X\times X$, the set
$E_0=\{y\in Y:\exists x\in D\;\;(x,y)\in E\mbox{ or }(y,x)\in E\}$ is meager in $Y$.
For every ordinal $\alpha<\kappa$ the set $G_\alpha$ is nowhere meager in $Y$,
which allows us to find a point $y_\alpha\in U_\alpha\cap G_\alpha\setminus (X_\alpha\cup E_0)$.
Then $B=D\cup\{y_\alpha\}_{\alpha<\kappa}$ is a nowhere meager $E$-free set in $X$.


\begin{thebibliography}{ChG}


\bibitem{BKT} T.~Banakh, Z.~Koszto\l owicz, S.~Turek,
 {\em Hereditarily supercompact spaces}, preprint\newline (http://arxiv.org/abs/1301.5297).

\bibitem{BanZdo}  T.~Banakh, L.~Zdomskyy, {\em Coherence of Semifilters: a survey}, in: Selection Principles and Covering Properties in Topology (L.~Kocinac ed.), Quaderni di Matematica. {\bf 18} (2006), 53--99.

\bibitem{BarJud95} T.~Bartoszy\'nski, H.~Judah, {\it Set theory. On the structure of the real line.} A. K. Peters, Ltd., Wellesley, MA, 1995. xii+546 pp.

\bibitem{BJS} T.~Bartoszynski, H.~Judah, S.~Shelah,  {\em The Cichon diagram}, J. Symbolic Logic, {\bf 58} (1993) 401--423.

\bibitem{Bla10} A.~Blass, \emph{Combinatorial cardinal characteristics of the continuum} in: \textit{Handbook of Set Theory} (M.\ Foreman, A.\ Kanamori, and M.\ Magidor, eds.), Springer, 2010, pp. 395--491.

\bibitem{Geschke1} S.~Frick, S.~Geschke, {\em Basis theorems for continuous n-colorings},  J. Combin. Theory Ser. A {\bf 118} (2011), 1334--1349.

\bibitem{Geschke2} S.~Geschke, {\em Weak Borel chromatic numbers}, MLQ Math. Log. Q. {\bf 57} (2011), 5--13.

\bibitem{Ke} A.~Kechris,
{\em Classical descriptive set theory}, Springer-Verlag, New York, 1995.

\bibitem{Kubis} W.~Kubis, {\em Perfect cliques and $G_\delta$ colorings of Polish spaces},
 Proc. Amer. Math. Soc. {\bf 131} (2003), 619--623.

%\bibitem{Mat77} Mathias, A.R.D., {\it Happy families,}  Annals of Mathematical  Logic  \textbf{12} (1977),  59--111.

\bibitem{NPS}  L.~Newelski, J.~Pawlikowski, W.~Seredy\'nski, {\em Infinite free sets for small measure set mappings}, Proc. Amer. Math. Soc. {\bf 100} (1987), 335--339.

\bibitem{SS} S.~Solecki, O.~Spinas, {\em Dominating and unbounded free sets},
 J. Symbolic Logic {\bf 64} (1999), 75--80.

\bibitem{Solovay} R.~Solovay, {\em A model of set-theory in which every set of reals is Lebesgue measurable}, Ann. of Math. (2) {\bf 92} (1970) 1--56.

\bibitem{Tal82} M.~Talagrand, {\it Filtres: Mesurabilit\'{e}, rapidit\'{e},
 propri\'{e}t\'{e} de Baire forte}, Studia Math. \textbf{74} (1982), 283--291.

\end{thebibliography}
\end{document}